\documentclass[10pt]{article}%
\usepackage{a4}%
\usepackage[final]{optional}
\usepackage[centertags]{amsmath}%
\usepackage{eucal,dsfont,accents,bbm,amsfonts,amssymb,amsthm,amsopn,wasysym,paralist,latexsym,url}%
\usepackage[usenames]{color}
\definecolor{darkgreen}{rgb}{0,0.6,0}
\definecolor{darkred}{rgb}{0.8,0,0}
\usepackage[colorlinks, citecolor=darkgreen, linkcolor=darkred]{hyperref}
\title{On Type I Singularities in Ricci flow}
\author{Joerg Enders, Reto M\"uller and Peter M.\ Topping}
\date{}
\theoremstyle{plain}
\newtheorem{lemma}{Lemma}[section]
\newtheorem{thm}[lemma]{Theorem}
\newtheorem{prop}[lemma]{Proposition}

\theoremstyle{definition}
\newtheorem{defn}[lemma]{Definition}

\newtheorem{rmk}[lemma]{Remark}

\numberwithin{equation}{section}

\newcommand{\m}{\ensuremath{{\cal M}}}
\newcommand{\ep}{\varepsilon}
\newcommand{\R}{\ensuremath{{\mathbb R}}}
\newcommand{\N}{\ensuremath{{\mathbb N}}}

\newcommand{\grad}{\nabla}
\newcommand{\del}{\partial}
\newcommand{\Rm}{{\mathrm{Rm}}}
\newcommand{\Ric}{{\mathrm{Ric}}}
\newcommand{\RS}{{\mathrm{R}}}

\begin{document}
\maketitle
\parskip=10pt

\begin{abstract}
We define several notions of singular set for Type I Ricci flows and
show that they all coincide. In order to do this, we prove that
blow-ups around singular points converge to \emph{nontrivial}
gradient shrinking solitons, thus extending work of Naber
\cite{Naber}. As a by-product we conclude that the volume of a
finite-volume singular set vanishes at the singular time.

We also define a notion of density for Type I Ricci flows and use it
to prove a regularity theorem reminiscent of White's partial
regularity result for mean curvature flow \cite{White05}.
\end{abstract}

%------------------------------------------------
\section{Introduction}
\label{intro}
A family $(\m^n,g(t))$ of smooth complete Riemannian $n$-manifolds
satisfying Hamilton's Ricci flow \cite{ham3D},
\begin{equation}\label{RicciFlow}
\frac{\del}{\del t} g = -2 \Ric_{g(t)},
\end{equation}
on a finite time interval $[0,T)$, $T<\infty$, is called a
\textbf{Type I Ricci flow} if there exists a constant $C>0$ such
that for all $t\in[0,T)$
\begin{equation}\label{typeIeq}
\sup_{\m}|\Rm_{g(t)}|_{g(t)}\le\frac{C}{T-t}.
\end{equation}
Such a solution is said to develop a \textbf{Type I singularity} at
time $T$ (and $T$ is called a \textbf{Type I singular time}) if it
cannot be smoothly extended past time $T$. It is well known that
this is the case if and only if
\begin{equation}\label{Tsingular}
\limsup_{t\nearrow T}\sup_{\m}|\Rm_{g(t)}|_{g(t)}=\infty,
\end{equation}
see \cite{ham3D} for compact and \cite{shi} for complete flows. Here
$\Rm_{g(t)}$ denotes the Riemannian curvature tensor of the metric
$g(t)$. The main examples of Ricci flow singularities are of Type I,
in particular the important neck-pinch singularity modelled on a
shrinking $n$-dimensional cylinder (cf. \cite{AngenentKnopf1,
AngenentKnopf2}) and singularities modelled on flows starting at a
positive Einstein metric or more general at a gradient shrinking
soliton with bounded curvature (see Section \ref{blowup}). Only very
few rigorous examples of finite time singularities which are not of
Type I (i.e. Type II) are known (cf. \cite{DaskalopoulosHamilton,
GuZhu}).

Since the manifolds $(\m,g(t))$ have bounded curvatures in the Type
I case (\ref{typeIeq}), the parabolic maximum principle applied to
the evolution equation satisfied by $|\Rm|^2$ shows that
(\ref{Tsingular}) is equivalent to
\begin{equation}\label{typeIlowerbound}
\sup_{\m}|\Rm_{g(t)}|_{g(t)}\ge\frac{1}{8(T-t)} \quad \text{ for all
} t\in[0,T).
\end{equation}
This motivates the following definitions.
\begin{defn}
A quantity $A(t)$ is said to \textbf{blow up at the Type I rate} as
$t\to T$ if there exist constants $C\geq c >0$ such that
$\frac{c}{T-t}\leq A(t)\leq \frac{C}{T-t}$ for all $t\in[T-c,T)$.
\end{defn}
\begin{defn}\label{TIsingpoint}
A space-time sequence $(p_i,t_i)$ with $p_i\in\m$ and $t_i\nearrow
T$ in a Ricci flow is called an \textbf{essential blow-up sequence}
if there exists a constant $c>0$ such that
\begin{equation*}
|\Rm_{g(t_i)}|_{g(t_i)}(p_i)\ge\frac{c}{T-t_i}.
\end{equation*}
A point $p\in\m$ in a Type I Ricci flow is called a (general)
\textbf{Type I singular point} if there exists an essential blow-up
sequence with $p_i\to p$ on $\m$. We denote the set of all Type I
singular points by $\Sigma_I$.
\end{defn}
\begin{rmk}
If a solution to (\ref{RicciFlow}) develops a Type I singularity at
time $T$, the existence of an essential blow-up sequence follows
from (\ref{typeIlowerbound}). If in addition $\m$ is compact, Type I
singular points always exist. In the noncompact case the Type I
singular set $\Sigma_I$ may be empty if the singularity \emph{forms
at spatial infinity}. An example where this happens could be a
cylinder $S^{n-1}\times \R$ with radius larger than 1 in the center
and tapering down to 1 at the ends. Flowing this under Ricci
flow would lead to a first blow-up at spatial infinity.
\end{rmk}
A conjecture, normally attributed to Hamilton, is that a suitable
blow-up sequence for a Type I singularity converges to a nontrivial
gradient shrinking soliton \cite{formations}. In the case where the
blow-up limit is compact, this conjecture was confirmed by Sesum
\cite{Sesum}. In the general case, blow-up to a gradient shrinking
soliton was proved by Naber \cite{Naber}. However, it remained an
open question whether the limit soliton Naber constructed is
nontrivial, as mentioned for example in \cite[Section 3.2]{Cao}. In
particular, one might think that the limit could be flat if all
essential blow-up sequences converged ``slowly'' to $p$ so that the
curvature disappears at infinity after parabolically rescaling. One
of the goals of this article is to rule out this possibility.
More precisely, we prove the following theorem.
\begin{thm}\label{mainthm1}
Let $(\m^n,g(t))$ be a Type I Ricci flow on $[0,T)$ and suppose
$p\in\Sigma_I$ is a Type I singular point as in Definition
\ref{TIsingpoint}. Then for every sequence $\lambda_j\to\infty$, the
rescaled Ricci flows $(\m,g_j(t),p)$ defined on $[-\lambda_jT,0)$ by
$g_j(t):=\lambda_j g(T+\frac{t}{\lambda_j})$ subconverge to a
normalized nontrivial gradient shrinking soliton in canonical form.
\end{thm}
We now turn to the relationship between the set $\Sigma_I$ of Type I
singular points and other notions of singular sets, starting with
the set of \emph{special} Type I singular points $\Sigma_s$ defined
as follows.
\begin{defn}\label{specialsingpoint}
A point $p\in\m$ in a Type I Ricci flow is called a \textbf{special
Type I singular point} if there exists an essential blow-up sequence
$(p_i,t_i)$ with $p_i = p$ for all $i\in\N$. The set of all such
points is denoted by $\Sigma_s$. Moreover, we denote by
$\Sigma_{\Rm} \subseteq \Sigma_s$ the set of points $p\in\m$ for
which $|\Rm_{g(t)}|_{g(t)}(p)$ blows up at the Type I rate as
$t\rightarrow T$.
\end{defn}
For mean curvature flow, Le-Sesum \cite{Hbounded} proved that the
mean curvature (rather than the second fundamental form) must be
unbounded at a Type I singular time. It is not surprising and known
to some Ricci flow experts that a similar result is true for the
Ricci flow: if $T$ is a Type I singular time, then the \emph{scalar}
curvature $R_{g(t)}$ is unbounded as $t\to T$. We make the following
definition.
\begin{defn}\label{Rsingpoint}
The set $\Sigma_R$ is defined to be the set of points $p\in\m$ for
which $R_{g(t)}(p)$ blows up at the Type I rate as $t\rightarrow T.$
\end{defn}
Instead of defining more restrictive singular sets, one can also
think of a priori larger sets of singular points, for example the
set consisting of points $p\in\m$ where $|\Rm_{g(t)}|_{g(t)}(p)$ is
unbounded as $t\to T$ but possibly blows up at a rate smaller than
the Type I rate, e.g.\ like $\frac{1}{(T-t)^\alpha}$ for some
$\alpha<1$. A priori it is not clear whether (in the presence of a
Type I singularity) such slowly forming singularities may exist in
another part of the manifold, in particular since they cannot be
observed by a blow-up argument analogous to Theorem \ref{mainthm1}.
The following is the most general, natural definition of the
singular set.
\begin{defn}\label{singpoint}
We call $p\in\m$ a \textbf{singular point} if there does not exist
any neighbourhood $U_p \ni p$ on which $|\Rm_{g(t)}|_{g(t)}$ stays
bounded as $t\to T$. The set of all singular points in this sense is
denoted by $\Sigma$.
\end{defn}
From the above definitions it is clear that
\begin{equation}\label{nestedsets}
\Sigma_R \subseteq \Sigma_{\Rm} \subseteq \Sigma_s \subseteq
\Sigma_I \subseteq \Sigma.
\end{equation}
For mean curvature flow with $H>0$, Stone \cite{Stone} showed that the
(corresponding) notions of singular sets $\Sigma_s$, $\Sigma_I$ and
$\Sigma$ agree. The same is true for the Ricci flow; in fact, we
show the slightly stronger result that all the singular sets defined
above are identical.
\begin{thm}\label{mainthm2}
Let $(\m^n,g(t))$ be a Type I Ricci flow on $[0,T)$ with singular
time $T$. Then $\Sigma \subseteq \Sigma_R$, i.e.\ all the different
notions of nested singular sets in (\ref{nestedsets}) agree.
\end{thm}
In particular, this shows that for a Type I Ricci flow there cannot
exist singular points where $R_{g(t)}$ stays bounded or blows up at
a rate smaller than the Type I rate as $t\rightarrow T$. As a
corollary, we conclude that the singular set $\Sigma$ has
asymptotically vanishing volume if its volume is bounded initially.
\begin{thm}\label{mainthm3}
Let $(\m^n,g(t))$ be a Type I Ricci flow on $[0,T)$ with singular
time $T$ and singular set $\Sigma$ as in Definition \ref{singpoint}.
If $\mathrm{Vol}_{g(0)}(\Sigma)<\infty$ then
\begin{equation*}
\mathrm{Vol}_{g(t)}(\Sigma)\xrightarrow{t\rightarrow T} 0.
\end{equation*}
\end{thm}
\begin{rmk}
A shrinking cylinder $\mathbb{S}^m\times \mathbb{R}^{n-m}$, $n> m\ge
2$, shows that the condition $\mathrm{Vol}_{g(0)}(\Sigma)<\infty$ is
necessary.
\end{rmk}
The paper is organized as follows. In Section \ref{blowup}, we prove
Theorem \ref{mainthm1}. The methods we are using strongly rely on
Perelman's results \cite{P1}. First, we recall Naber's result
\cite{Naber} that for \emph{any} point $p\in\m$ the rescaled flows
$g_j(t)$ as defined in Theorem \ref{mainthm1} converge to a gradient
shrinking soliton (Theorem \ref{limsolthm}). This is based on a
version of Perelman's reduced length and volume based at the
singular time, developed independently by the first author
\cite{Enders_Thesis} and Naber \cite{Naber}. For completeness, we
sketch the main arguments of the proof. We then use Perelman's
pseudolocality theorem \cite{P1} to show that the limit soliton must
be nontrivial if $p\in\Sigma_I$ is a Type I singular point. This
completes the proof of Theorem \ref{mainthm1}. In Section
\ref{singsize}, we prove Theorem \ref{mainthm2}. The argument is
based again on Perelman's pseudolocality result as well as a strong
rigidity result for gradient shrinking solitons, which can be found in
Pigola-Rimoldi-Setti \cite{PRS09}. As a corollary, we obtain a proof
of Theorem \ref{mainthm3}. Finally, in the last section we define a
\textbf{density function} $\theta_{p,T}$ for Type I Ricci flows,
related to the central density for gradient shrinking solitons
defined by Cao-Hamilton-Ilmanen \cite{CaoHamiltonIlmanen}, and prove
a regularity type theorem (Theorem \ref{gapthm}) resembling White's
regularity result for mean curvature flow \cite{White05} and Ni's
regularity theorem for Ricci flow \cite{Ni}. The
proof of this result uses a gap theorem of Yokota \cite{Yokota}.

Le and Sesum have also been studying properties of the scalar
curvature at a Ricci flow singularity. In the current version of their paper
\cite{LeSesum}, they observe how the arguments from our paper in fact exclude Type I
singularity formation for compact manifolds under the assumption of an
integral (rather than pointwise) scalar curvature bound.
\paragraph{Acknowledgements:} All authors were partially
supported by The Leverhulme Trust. RM was partially supported by
FIRB Ideas ``Analysis and Beyond''.
%------------------------------------------------
\section{Blow-up to nontrivial gradient shrinking solitons}
\label{blowup}
Before we start proving Theorem \ref{mainthm1}, let us briefly
recall some basic definitions and facts about gradient shrinking
solitons as well as the essential definitions and results from the
first author \cite{Enders_Thesis} and Naber \cite{Naber}.
\begin{defn}
A triple $(\m^n,g,f),$ where $(\m,g)$ is a complete $n$-dimensional
Riemannian manifold and $f: \m\rightarrow \R$ a smooth function, is
called \textbf{gradient shrinking soliton} if
\begin{equation*}
\Ric_g + \grad^g\grad f =\frac{1}{2}g.
\end{equation*}
\end{defn}
It is well known, that we can \textbf{normalize} $f$ on a gradient
shrinking soliton by setting
\begin{equation}\label{normeq}
R_g+|\grad f|_g^2-f = 0.
\end{equation}
It follows from (\ref{normeq}) and the fact that $R\ge 0$ (cf. e.g.
\cite{Zhang}), that $\grad f$ is a complete vector field. Letting
$T>0$ and considering the diffeomorphisms $\phi_t$ of $\m$ generated
by $\frac{1}{T-t}\grad f$ with $\phi_{T-1}=id,$ we obtain from the
definition of gradient shrinking soliton above a corresponding Ricci
flow $g(t)=(T-t)\phi_t^* g$ on $(-\infty,T)$ with
$(\m,g(T-1))=(\m,g).$ Canonically defining time-dependent functions
by $f(t):=\phi_t^* f,$ the flow satisfies
\begin{equation}
\Ric_{g(t)} + \grad^{g(t)}\grad f(t) = \frac{1}{2(T-t)} g(t) \qquad
\text{ and }\qquad \frac{\del}{\del t} f(t) = |\grad f(t)|_{g(t)}^2.
\label{cansoleq}
\end{equation}
We call a Ricci flow $(\m,g(t),f(t))$ on $(-\infty,T)$ with smooth
functions $f(t):\m\rightarrow \R$ satisfying (\ref{cansoleq}) a
\textbf{gradient shrinking soliton in canonical form}.

Let $(\m^n,g(t))$ be a (connected) Type I Ricci flow on $[0,T)$ as
defined in the Section \ref{intro}. For fixed $(p,t_0)\in \m\times
[0,T)$ and all $(q,\bar{t})\in \m\times [0,t_0],$ Perelman's
\textbf{reduced distance} (in forward time notation) is defined by
\begin{equation*}
l_{p,t_0}(q,\bar{t}):=\inf_{\gamma}\left\{\frac{1}{2\sqrt{t_0-\bar{t}}}
\int_{\bar{t}}^{t_0}\sqrt{t_0-t}\big(|\dot{\gamma}(t)|^2+R_{g(t)}(\gamma(t))\big)dt\right\},
\end{equation*}
where the infimum is taken over all curves
$\gamma:[\bar{t},t_0]\rightarrow \m \text{ with }
\gamma(t_0)=p,\gamma(\bar{t})=q.$ The corresponding \textbf{reduced
volume} is
\begin{equation*}\tilde{V}_{p,t_0}(\bar{t}):=\int_{\m}
v_{p,t_0}(q,\bar{t})dvol_{g(\bar{t})}(q),
\end{equation*}
where
\begin{equation*}
v_{p,t_0}(q,\bar{t}):=\big(4\pi(t_0-\bar{t})\big)^{-\frac{n}{2}}e^{-l_{p,t_0}(q,\bar{t})}.
\end{equation*}
We will use the following two results from \cite{Enders_Thesis}
(restricted here to the Type I case):
\begin{lemma}[Enders \cite{Enders_Thesis}, Theorem 3.3.1]\label{lsing}
Let $(\m^n,g(t))$ be a (connected) Type I Ricci flow on $[0,T)$,
$p\in \m$ and $t_k\nearrow T.$ Then there exists a locally Lipschitz
function
\begin{equation*}
l_{p,T}:\m\times (0,T)\rightarrow \mathbb{R},
\end{equation*}
which
is a subsequential limit
\begin{equation*}
l_{p,t_k}\xrightarrow{C^0_{loc}(\m\times (0,T))} l_{p,T}
\end{equation*}
and which for all \mbox{$(q,\bar{t})\in \m\times(0,T)$} satisfies
\begin{equation*}
-\frac{\del}{\del\bar{t}}l_{p,T}(q,\bar{t})-\Delta_{g(\bar{t})}
l_{p,T}(q,\bar{t})+|\nabla
l_{p,T}(q,\bar{t})|_{g(\bar{t})}^2-\RS_{g(\bar{t})}(q)+\frac{n}{2(T-\bar{t})}\ge
0
\end{equation*}
in the sense of distributions. Equivalently,
\begin{equation*}
\square^*_{g(\bar{t})} v_{p,T}(q,\bar{t})\le 0,
\end{equation*} where
\begin{equation*}
\square^*_{g(t)}:=-\frac{\del}{\del t}-\Delta_{g(t)}+\RS_{g(t)}
\end{equation*}
denotes the formal adjoint of the heat operator under the Ricci
flow, and
\begin{equation*}
v_{p,T}(q,\bar{t}):=\big(4\pi(T-\bar{t})\big)^{-\frac{n}{2}}e^{-l_{p,T}(q,\bar{t})}.
\end{equation*}
\end{lemma}
\begin{defn}\label{deflvsing}
We define $l_{p,T}$ as in Lemma \ref{lsing} to be \textbf{a reduced
distance based at the singular time $(p,T).$} Moreover, the
corresponding
\begin{equation*}
\tilde{V}_{p,T}(\bar{t}):=\int_\m v_{p,T}(q,\bar{t})
dvol_{g(\bar{t})}(q)
\end{equation*}
is denoted \textbf{a reduced volume based at the singular time
$(p,T)$} with $v_{p,T}$ being \textbf{a reduced volume density based
at the singular time $(p,T)$}.
\end{defn}
The next result states that similarly to Perelman's reduced volume,
any reduced volume based at singular time is also a monotone
quantity.
\begin{lemma}[Enders \cite{Enders_Thesis}, Theorem 3.4.3]\label{vsing}
Under the assumptions as in Definition \ref{deflvsing} we have
\begin{enumerate}[(i)]
  \item $\frac{d}{d\bar{t}}\tilde{V}_{p,T}(\bar{t})\ge 0,$
  \item $\lim_{\bar{t}\nearrow T}\tilde{V}_{p,T}(\bar{t})\le 1,$
  \item If $\tilde{V}_{p,T}(\bar{t}_1)=\tilde{V}_{p,T}(\bar{t}_2)$
  for $0<\bar{t}_1<\bar{t}_2<T,$ then $(\m,g(t),l_{p,T}(\,\cdot\,,t))$
  is a normalized gradient shrinking soliton in canonical form.
\end{enumerate}
\end{lemma}
Similar results to Lemma \ref{lsing} and Lemma \ref{vsing} have been
independently obtained in \cite{Naber}. We restate the estimates
derived there in the following adapted form.
\begin{lemma}[Naber \cite{Naber}, Proposition 3.6]\label{lest}
Let $(\m^n,g(t))$ be a (connected) Type I Ricci flow on $[0,T)$, and
let $(p,t_0)\in \m\times [0,T).$ Then there exist $K>0$ (only
dependent on $n$ and the Type I constant $C$) such that for all
$(q,\bar{t})\in \m\times (0,T)$
\begin{enumerate}[(i)]
\item $\frac{1}{K}\Big(1+\frac{d_{\bar{t}}(p,q)}{\sqrt{t_0-\bar{t}}}\Big)^2
-K\le l_{p,t_0}(q,\bar{t})\le
K\Big(1+\frac{d_{\bar{t}}(p,q)}{\sqrt{t_0-\bar{t}}}\Big)^2,$
\item $|\nabla l_{p,t_0}(q,\bar{t})|_{g(\bar{t})}(q)\le\frac{K}{\sqrt{t_0-\bar{t}}}
\Big(1+\frac{d_{\bar{t}}(p,q)}{\sqrt{t_0-\bar{t}}}\Big),$
\item $|\frac{\del}{\del\bar{t}}l_{p,t_0}(q,\bar{t})|_{g(\bar{t})}(q)
\le\frac{K}{t_0-\bar{t}}\Big(1+\frac{d_{\bar{t}}(p,q)}{\sqrt{t_0-\bar{t}}}\Big)^2.$
\end{enumerate}
\end{lemma}
We now show that parabolic rescaling limits in a Type I Ricci flow
(around any point $p\in \m$ at the singular time $T$) have a
gradient shrinking soliton structure. For completeness, we reprove
this result which was first obtained in \cite{Naber}.
\begin{thm}[cf. Naber \cite{Naber}, Theorem 1.5]\label{limsolthm}
Let $(\m^n,g(t),p),\,t\in [0,T),\,p\in \m$ be a pointed Type I Ricci
flow, and $\lambda_j \nearrow \infty.$ Then any pointed
Cheeger-Gromov-Hamilton limit flow $(\m^n_{\infty},
g_{\infty}(t),p_\infty),$ $\,t\in (-\infty,0),$ of the parabolically
rescaled Ricci flows $g_j(t):=\lambda_j g(T+\frac{t}{\lambda_j})$ is
a normalized gradient shrinking soliton in canonical form.
\end{thm}
\begin{proof}
Because of the Type I curvature bound, we have at any $x\in \m$ that
\begin{equation}\label{gjbound}
\begin{aligned}
|\Rm_{g_j(t)}|_{g_j(t)}(x)&=\frac{1}{\lambda_j}
|\Rm_{g(T+\frac{t}{\lambda_j})}|_{g(T+\frac{t}{\lambda_j})}(x)\\
&\le \frac{C}{\lambda_j\big(T-(T+\frac{t}{\lambda_j})\big)}
=\frac{C}{-t}.
\end{aligned}
\end{equation}
This gives a uniform curvature bound on compact subsets of
$(-\infty,0).$ Together with Perelman's no local collapsing theorem
(which also holds for complete $\m$ because of the uniform lower
bound on the reduced volume as described below), we can use the
Cheeger-Gromov-Hamilton Compactness Theorem \cite{CGHcomp} to
extract from the sequence $(\m,g_j(t),p)$ a complete pointed
subsequential limit Ricci flow $(\m_{\infty},
g_{\infty}(t),p_{\infty})$ on $(-\infty,0),$ which is still Type I.

Now, let $l_{p,T}$ be any reduced distance based at the singular
time $(p,T)$ for the Ricci flow $(\m,g(t))$ on $[0,T)$ as defined above.
For each $(q,\bar{t})\in \m\times
(-\infty,0),$ consider for large enough~$j$
\begin{equation}\label{ljdef}
l^j_{p,0}(q,\bar{t}):=l_{p,T}(q,T+\tfrac{\bar{t}}{\lambda_j}),
\end{equation}
which is a reduced distance based at the singular time $(p,0)$ for
the rescaled Ricci flow $(\m,g_j(t))$ on $[-\lambda_j T,0)$ because
of the scaling properties of the reduced distance. The corresponding
reduced volumes are then related by
\begin{equation*}
\tilde{V}^j_{p,0}(\bar{t})=\tilde{V}_{p,T}(T+\tfrac{\bar{t}}{\lambda_j}),
\end{equation*}
and we can conclude, using also Lemma \ref{vsing}, that
\begin{equation*}
\tilde{V}^j_{p,0}\xrightarrow{j\rightarrow\infty} \lim_{t\nearrow
T}\tilde{V}_{p,T}(t)\,\in\,(0,1]
\end{equation*} uniformly on compact subsets of
$(-\infty,0).$

The uniform estimates in Lemma \ref{lest} hold for $l_{p,T}$ by
construction, and hence by (\ref{ljdef}) for each $l_{p,0}^j.$ Note
that by (\ref{gjbound}) they have the same Type I bound $C$. Hence
we can conclude that there exists a locally Lipschitz function
$l^\infty_{p_\infty,0}$ on the limit manifold
$\m_\infty\times(-\infty,0),$ such that
\begin{equation*}
l^j_{p,0}\xrightarrow{C_{loc}^0}l^\infty_{p_\infty,0}.
\end{equation*}
Since its corresponding formal reduced volume
$V^\infty_{p_\infty,0}$ is constant, we can conclude as in the proof
of Lemma \ref{vsing} (iii) that $(\m_{\infty},
g_{\infty}(t),l^\infty_{p_\infty,0}(\,\cdot\,,t))$ is a normalized
gradient shrinking soliton in canonical form.
\end{proof}
To obtain a complete proof of Theorem \ref{mainthm1}, it remains to
show that for Type I singular points $p\in\Sigma_I$ the rescaling
limit flow $(\m_{\infty}, g_{\infty}(t))$ in Theorem \ref{limsolthm}
is nontrivial and hence a suitable singularity model.

Our proof is based on Perelman's pseudolocality theorem, which
states the following.
\begin{prop}[Perelman \cite{P1}, Theorem 10.3]\label{pseudolocality}
There exist $\ep,\delta>0$ depending on $n$ with the following
property. Suppose $g(t)$ is a complete Ricci flow with bounded
curvature on an $n$-dimensional manifold $\m^n$ for $t\in[0,(\ep
r_0)^2)$. Moreover, suppose that $r_0>0$, $p\in\m$ and assume that
at $t=0$ we have $|\Rm_{g(0)}|\leq r_0^{-2}$ in $B_{g(0)}(p,r_0)$
and $\mathrm{Vol}_{g(0)}\big(B_{g(0)}(p,r_0)\big)\geq
(1-\delta)\omega_n r_0^n$, where $\omega_n$ is the volume of the
unit ball in $\mathbb{R}^n$. Then there holds the following estimate
\begin{equation}\label{plestimate}
|\Rm_{g(t)}|(x)\leq (\ep r_0)^{-2}, \quad \text{for } 0\leq t<(\ep
r_0)^2, x\in B_{g(t)}(p,\ep r_0).
\end{equation}
\end{prop}
\begin{rmk}\label{pseudoimprove}
Note that by choosing a smaller $\ep$, estimate (\ref{plestimate})
holds for $x\in B_{g(0)}(p,\ep r_0)$. This follows directly from the following lemma, variants of which can be found elsewhere, for example \cite{formations}.
\end{rmk}
\begin{lemma}\label{distlemma}
Suppose that $g(t)$ is a Ricci flow on a manifold $\m^n$ for $t\in
[0,T]$. Suppose further that for some $p\in\m$ and $r>0$, we have
$B_{g(t)}(p,r)\subset\subset\m$, and $|\Ric|\leq M$ on
$B_{g(t)}(p,r)$ for each $t\in [0,T]$. Then
\begin{equation*}
B_{g(0)}(p,e^{-Mt}r)\subset B_{g(t)}(p,r),
\end{equation*}
for all $t\in [0,T]$.
\end{lemma}
\begin{proof}
Let $\sigma\in (0,1)$ be arbitrary. It suffices to show that
\begin{equation*}
\overline{B_{g(0)}(p,e^{-Mt}\sigma r)}\subset B_{g(t)}(p,r),
\end{equation*}
for each $t\in [0,T]$. Clearly this is true for $t=0$; suppose it
fails for some larger $t=t_0\in (0,T]$. Without loss of generality,
$t_0$ is the least such time.

Pick a minimizing geodesic $\gamma$, with respect to $g(0)$, from
$p$ to a point $y\in\m$ with $d_{g(0)}(p,y)=e^{-Mt_0}\sigma r$ and
$d_{g(t_0)}(p,y)=r$.

Then for all $t\in [0,t_0)$, $\gamma$ lies within
$B_{g(0)}(p,e^{-Mt}\sigma r)\subset B_{g(t)}(p,r)$, and hence (by
hypothesis) $|\Ric|\leq M$ on $\gamma$ over this range of times.

But then $Length_{g(t)}(\gamma)\leq e^{Mt}Length_{g(0)}(\gamma)$ for
all $t\in [0,t_0)$, and hence
\begin{equation*}
Length_{g(t_0)}(\gamma)\leq e^{Mt_0}e^{-Mt_0}\sigma r = \sigma r<r,
\end{equation*}
and this implies $d_{g(t_0)}(p,y)<r$, a contradiction.
\end{proof}
We are now ready to prove that the blow-up limit is nontrivial.
\begin{proof}[Proof of Theorem \ref{mainthm1}]
Assume that $g_\infty(t)$ is flat for all $t<0$. In particular,
$g_\infty(t)$ is independent of time, and we denote it by $\hat{g}$.

Take $r_0>0$ smaller than the injectivity radius of $\hat{g}$ at $p_\infty$
to ensure that $B_{\hat{g}}(p_\infty,r_0)$ is a Euclidean ball. By the
Cheeger-Gromov-Hamilton convergence, taking $j$ large enough, $B_{g_j(-(\ep
r_0)^2)}(p,r_0)$ is as close as we want to a Euclidean ball,
where $\ep$ is chosen as in Proposition \ref{pseudolocality}. In
particular, we may fix $j$ sufficiently large such that
$B_{g_j(-(\ep r_0)^2)}(p,r_0)$ satisfies the conditions of
Proposition \ref{pseudolocality} and hence, using also Remark
\ref{pseudoimprove},
\begin{equation}\label{plestim2}
|\Rm_{g_j(t)}|_{g_j(t)}(x)\leq (\ep r_0)^{-2} \quad \text{for }
-(\ep r_0)^2\leq t < 0, x\in B_{g_j(-(\ep r_0)^2)}(p,\ep r_0).
\end{equation}
On the other hand, since $p\in\Sigma_I,$ there exists an essential
blow-up sequence $(p_i,t_i)$ with $p_i\rightarrow p$ such that for a
constant $c>0$ as in Definition \ref{TIsingpoint}
\begin{equation*}
|\Rm_{g_j(\lambda_j(t_i-T))}|_{g_j(\lambda_j(t_i-T))}(p_i)\geq
\frac{c}{\lambda_j(T-t_i)}.
\end{equation*}
For $i$ large enough, this contradicts (\ref{plestim2}). Thus
$g_\infty(t)$ cannot be flat.
\end{proof}
Nontrivial gradient shrinking solitons also arise as blow-down limits
of certain ancient Ricci flow solutions (which are singularity models)
as shown by Perelman \cite{P1} in 3 dimensions, and recently by Cao and
Zhang \cite{CaoZhang} for higher dimensions in the Type I case.
%------------------------------------------------
\section{Singular sets}
\label{singsize}
A crucial ingredient for the theorems proved in this and the next
section is the following rigidity result for gradient shrinking
solitons as for example shown in \cite{PRS09}.
\begin{lemma}[Pigola-Rimoldi-Setti \cite{PRS09}, Theorem 3]\label{rigidity}
Let $(\m^n,g,f)$ be a complete gradient shrinking soliton. Then the
scalar curvature $R_g$ is nonnegative, and if there exists a point
$p\in\m$ where $R_g(p)=0$, then $(\m,g,f)$ is the Gaussian soliton,
i.e. isometric to flat Euclidean space $(\R^n,g_{\R^n})$.
\end{lemma}
We use this lemma to prove Theorem \ref{mainthm2}, i.e. that $\Sigma
\subseteq \Sigma_R$. As a first step, we show that the Type I
singular set $\Sigma_I$ is characterized by the blow-up of the
scalar curvature at the Type I rate, i.e.\ $\Sigma_I = \Sigma_R$.
\begin{thm}\label{thmRblowup}
Let $(\m^n,g(t))$ be a Type I Ricci flow on $[0,T)$ with singular
time $T$, Type I singular set $\Sigma_I$ as in Definition
\ref{TIsingpoint} and $\Sigma_R$ as in Definition \ref{Rsingpoint}.
Then $\Sigma_I = \Sigma_R$.
\end{thm}
\begin{proof}
By definition, we know that $\Sigma_R \subseteq \Sigma_I$. For the
converse inclusion, assume that $p\in \m\setminus\Sigma_R$. Hence
there are $c_j\searrow 0$ and $t_j\in[T-c_j,T)$, such that
$R_{g(t_j)}(p)< \frac{c_j}{T-t_j}$. Let
$\lambda_j=(T-t_j)^{-1}\rightarrow\infty$ and rescale as in Theorem
\ref{limsolthm}, i.e.\ let $g_j(t):=\lambda_j
g(T+\frac{t}{\lambda_j})$ on $\m \times [-\lambda_jT,0)$. By Theorem
\ref{limsolthm}, $(\m,g_j(t),p)$ converge to a gradient shrinking
soliton in canonical form $(\m_\infty,g_\infty(t),p_\infty)$ on
$(-\infty,0)$ with
\begin{equation*}
R_{g_\infty(-1)}(p_\infty)=\lim_{j\to\infty}\lambda_j^{-1}R_{g(t_j)}(p)
\leq \lim_{j\to\infty} c_j=0.
\end{equation*}
By Lemma \ref{rigidity}, $(\m_\infty,g_\infty(-1))=(\m_\infty,g_\infty(t))$
must be flat. But by Theorem \ref{mainthm1}, the limit soliton is nonflat for points
in the Type I singular set, so $p\in\m\setminus\Sigma_I$.
\end{proof}
To show that $\Sigma = \Sigma_I$, we prove the following regularity
type result.
\begin{thm}\label{regthm}
If $(\m^n,g(t))$ is a Type I Ricci flow on $[0,T)$ with singular
time $T$ and $p\in \m \setminus \Sigma_I$, then there exists a
neighbourhood $U_p \ni p$ such that the curvature is bounded
uniformly on $U_p \times [0,T)$. In particular, $p\in\m\setminus\Sigma$.
\end{thm}
\begin{proof}
Since $p\in\m\setminus\Sigma_I$, for any given $\lambda_j\rightarrow
\infty$ the rescaled metrics $g_j(t)=\lambda_j
g(T+\frac{t}{\lambda_j})$ converge to flat Euclidean space. As in
the proof of Theorem \ref{mainthm1}, for large enough $j\ge j_0$ the
conditions of the pseudolocality theorem, Proposition
\ref{pseudolocality}, and Lemma \ref{distlemma} with $r_0=1$ are
satisfied. Let $K:=\lambda_{j_0}$ and take $\ep>0$ to be as in the
pseudolocality theorem and Remark \ref{pseudoimprove} following it.
Then we conclude for $j=j_0$ as before
\begin{equation*}
|\Rm_{g_{j_0}(t)}|_{g_{j_0}(t)}(x)\leq \ep^{-2} \quad \text{for }
-\ep^2\leq t < 0, x\in B_{g_{j_0}(-\ep^2)}(p,\ep),
\end{equation*}
which is equivalent to
\begin{equation*}
|\Rm_{g(t)}|_{g(t)}\leq \frac{K}{\ep^2} \quad \text{for all } t \in
[T-\tfrac{\ep^2}{K},T)
\end{equation*}
on the neighbourhood $U_p:= B_{g(T-\frac{\ep^2}{K})}
(p,\frac{\ep}{\sqrt{K}})$ of $p$. The bound for times
$t<T-\frac{\ep^2}{K}$ follows trivially from the Type I condition
(\ref{typeIeq}).
\end{proof}
Combining Theorem \ref{thmRblowup} and Theorem \ref{regthm}, we have
proved Theorem \ref{mainthm2}. As a corollary, we obtain Theorem
\ref{mainthm3}, i.e. that the singular set $\Sigma$ has
asymptotically vanishing volume if
$\mathrm{Vol}_{g(0)}(\Sigma)<\infty$.
\begin{proof}[Proof of Theorem \ref{mainthm3}.]
By the bounded curvature assumption (\ref{typeIeq}) together with
the parabolic maximum principle applied to the evolution of
$R_{g(t)}$, there exists $\tilde{C}>0$ such that $\inf_\m
R_{g(t)}\geq -\tilde{C}$, $\forall t\in[0,T)$. Let $\Sigma_{R,k}$ be
defined by
\begin{equation*}
\Sigma_{R,k}:=\{p\in\m\,|\, R_{g(t)}(p)\ge\tfrac{1/k}{T-t},\,\forall
t\in(T-\tfrac{1}{k},T)\}\subseteq \Sigma_R = \Sigma
\end{equation*}
for $k\in\mathbb{N}$ and $\Sigma_{R,0}:=\emptyset$. We claim that on
$\Sigma_{R,k}$, we have for all $t\in[0,T)$
\begin{equation*}
\int_0^t R_{g(s)}ds \ge -\tilde{C}T +
\log\Big(\tfrac{1/k}{T-t}\Big)^{1/k}.
\end{equation*}
For $t\le T-\frac{1}{k}$, this follows from $\int_0^t Rds\ge
-\tilde{C}t\ge -\tilde{C}T$ and the fact that the $\log$-term is
nonpositive in this case. For $t\in(T-\frac{1}{k},T)$, we obtain by
definition of $\Sigma_{R,k}$
\begin{align*}
\int_0^t R_{g(s)}ds &= \int_0^{T-\frac{1}{k}} R_{g(s)}ds +
\int_{T-\frac{1}{k}}^t R_{g(s)}ds\\
&\ge -\tilde{C}(T-\tfrac{1}{k})+\int_{T-\frac{1}{k}}^t
\tfrac{1/k}{T-s}ds\\
&\ge -\tilde{C}T + \log\Big(\tfrac{1/k}{T-t}\Big)^{1/k}.
\end{align*}
Using $k^{1/k} \le 2$ for all $k\in\mathbb{N}$, we can now bound
volumes of subsets of $\Sigma_{R,k}$ at time $t$ in terms of their
volumes at time $0$ by computing
\begin{align*}
  \mathrm{Vol}_{g(t)}(\Sigma_{R,k}\setminus \Sigma_{R,k-1})&= \int_{\Sigma_{R,k}\setminus \Sigma_{R,k-1}} e^{-(\int_0^t
R_{g(s)}ds)}dvol_{g(0)} \\
   &\leq 2e^{\tilde{C}T}(T-t)^{1/k}\,
\mathrm{Vol}_{g(0)}(\Sigma_{R,k}\setminus \Sigma_{R,k-1}).
\end{align*}
We use this last estimate to conclude
\begin{align*}
\limsup_{t\to T}\mathrm{Vol}_{g(t)}(\Sigma) &= \limsup_{t\to
T}\sum_{k\in\mathbb{N}}\mathrm{Vol}_{g(t)}(\Sigma_{R,k}\setminus
\Sigma_{R,k-1})\\
&\le 2e^{\tilde{C}T} \lim_{t\to T}
\sum_{k\in\mathbb{N}}(T-t)^{1/k}\,\mathrm{Vol}_{g(0)}(\Sigma_{R,k}\setminus\Sigma_{R,k-1})\\
&=0,
\end{align*}
where the last line follows easily from the fact that
$\sum_{k\in\mathbb{N}}\mathrm{Vol}_{g(0)}(\Sigma_{R,k}\setminus\Sigma_{R,k-1})=\mathrm{Vol}_{g(0)}(\Sigma_R)<\infty.$
\end{proof}
%------------------------------------------------
\section{Density and regularity theorem}
\label{density}
In this section, we use the reduced volume based at the singular
time as reviewed in Section \ref{blowup} to define a density
function on the closure of space-time of Type I Ricci flows and
prove a regularity theorem. We first overcome the non-uniqueness
issue of the reduced distance based at the singular time. The functions
$l_{p,T}$ used in Section \ref{blowup} were subsequential limits and
depended on the choice of $\{t_k\}$ and a subsequence $\{t_{k_l}\}.$
We now denote such a choice of reduced distance based at the
singular time by $l_{p,T,\{t_{k_l}\}}$ to make the following
definition.
\begin{defn}\label{infdef}
Under the assumptions of Lemma \ref{lsing}, we define \textbf{the
reduced distance based at the singular time} by
\begin{equation*}
l_{p,T}:=\inf_{\{t_{k_l}\}}l_{p,T,\{t_{k_l}\}},
\end{equation*}
where the infimum is taken over all possible subsequences of all
possible sequences $t_k\nearrow T$ used to construct a reduced
distance. As in Definition \ref{deflvsing}, we correspondingly
denote \textbf{the reduced volume density} and \textbf{the reduced
volume based at the singular time} by $v_{p,T}$ and
$\tilde{V}_{p,T},$ respectively.
\end{defn}
Note that Lemma \ref{lest} implies that $l_{p,T}$ is well-defined
and locally Lipschitz because $l_{p,T,\{t_{k_l}\}}$ are uniformly
locally Lipschitz. The monotonicity in Lemma \ref{vsing} for the
redefined $\tilde{V}_{p,T}$ as above holds once we show the
following lemma.
\begin{lemma} Under the assumptions as in Lemma \ref{lsing}, we have
for $v_{p,T}$ in Definition \ref{infdef} that
\begin{equation*}
\square^*_{g(\bar{t})}v_{p,T}(q,\bar{t})\le 0
\end{equation*}
holds in the weak sense or sense of distributions.
\end{lemma}
\begin{proof}
We argue by contradiction: Assume there exists a (small) parabolic
cylinder $P=U\times[t_2,t_1)\subset \m\times (0,T),$ $U$ open, such
that for all $0\le\phi\in C^2_{cpt}(\m\times(0,T))$ with support in
$P$
\begin{equation*}
\iint_{P} v_{p,T}(q,t) \square_{g(t)}\phi(q,t)\, dvol_{g(t)}(q) dt> 0.
\end{equation*}
Inverting time ($\tau:=T-t$) implies that $-v_{p,T}$ is strictly
subparabolic in the weak sense of Friedman \cite{Friedman61}, and we
will apply his (strong) maximum principle several times to derive a
contradiction.

By Definition \ref{infdef},
$v_{p,T}:=\sup_{\{t_{k_l}\}}v_{p,T,\{t_{k_l}\}}.$ Let $\{t_{k_l}\}$
be any such subsequence, then $v_{p,T,\{t_{k_l}\}}\le v_{p,T},$ and
we know from Lemma \ref{lsing} that
\begin{equation*}
\square^*v_{p,T,\{t_{k_l}\}}\le 0
\end{equation*}
in the weak sense. Now let $\Gamma:=\bar{P}\backslash P$ and
$w_{\{t_{k_l}\}}$ be a weak solution to
\begin{equation*}
\left\{
\begin{array}{rcl}
  \square^* w_{\{t_{k_l}\}}&=& 0 \,\,\,\text{ in P }\\
  w_{\{t_{k_l}\}}|_{\Gamma} &=& v_{p,T,\{t_{k_l}\}}|_{\Gamma}.
\end{array}\right.
\end{equation*}
Hence,
\begin{equation*}
\square^* (v_{p,T,\{t_{k_l}\}}-w_{\{t_{k_l}\}})\le 0,
\end{equation*}
and the maximum principle implies
\begin{equation}\label{vpwp}
v_{p,T,\{t_{k_l}\}} \le w_{\{t_{k_l}\}} \text{ in }\bar{P}.
\end{equation}
Similarly, let $w$ be a weak solution to
\begin{equation*}
\left\{
\begin{array}{rcl}
  \square^* w&=& 0 \,\,\,\text{ in P }\\
  w|_{\Gamma} &=&v_{p,T}|_{\Gamma}.
\end{array}\right.
\end{equation*}
Since
\begin{equation*}
\square^* (w_{\{t_{k_l}\}}-w)=0
\end{equation*}
and $w_{\{t_{k_l}\}}|_\Gamma=v_{p,T,\{t_{k_l}\}}|_\Gamma\le
v_{p,T}|_\Gamma=w|_\Gamma,$ the maximum principle implies that
\begin{equation}\label{wpws}
w_{\{t_{k_l}\}} \le w \text{ in }\bar{P}.
\end{equation}
As $\{t_{k_l}\}$ was arbitrary, we conclude from (\ref{vpwp}) and
(\ref{wpws}) that
\begin{equation*}
v_{p,T} \le w \text{ in }\bar{P}.
\end{equation*}
Using $v_{p,T}|_\Gamma=w|_\Gamma$ and the maximum principle again,
this contradicts that by assumption
\begin{equation*}
\square^*(v_{p,T}-w)>0
\end{equation*}
in the weak sense in $P.$
\end{proof}
We now consider points in the closure of space-time, i.e. in
$\m\times [0,T],$ to include the singular time.
\begin{defn}
Let $(\m,g(t))$ be a Type I Ricci flow on $[0,T).$ For any
$(p,t_0)\in \m\times [0,T]$ we define \textbf{the density at
$(p,t_0)$ in the Ricci flow $(\m,g(t))$} by
\begin{equation*}
\theta_{p,t_0}:=\lim_{\bar{t}\nearrow t_0}
\tilde{V}_{p,t_0}(\bar{t})\in(0,1].
\end{equation*}
\end{defn}
Note that for the special case of gradient shrinking solitons,
Cao-Hamilton-Ilmanen \cite{CaoHamiltonIlmanen} suggest a ``central
density of a shrinker'' defined similarly.

If $T$ is the Type I singular time and $t_0<T$ it follows from the
properties of Perelman's reduced volume that $\theta_{p,t_0}=1$ for
any $p\in \m$ (in fact, without the Type I assumption). At the
singular time $t_0=T,$ the density carries information regarding the
structure of the singularity, namely the corresponding gradient
shrinking solitons one may obtain by taking a blow-up limit. We
prove the following regularity type result similar to White's local
regularity result for mean curvature flow \cite{White05}, where
instead of using the Gaussian density for the mean curvature flow we
use the density for Type I Ricci flows as defined above. A related
result is proved by Ni \cite{Ni} using a localized quantity.
\begin{thm}\label{gapthm}
Let $(\m^n,g(t))$ be a Type I Ricci flow on $[0,T)$ with singular
time $T$ and singular set $\Sigma$ as in Definition \ref{singpoint}.
Then $\theta_{p,T}=1$ if and only if $p\in \m \setminus \Sigma$. In
fact, there exists $\eta>0$ (only depending on $n$) such that if
$\theta_{p,T}>1-\eta$ for a Ricci flow as above, then $p\in \m
\setminus \Sigma$. Equivalently, if $\theta_{p,T}>1-\eta$, then
there exists a neighbourhood $U_p \ni p$ such that the curvature is
bounded uniformly on $U_p \times [0,T)$.
\end{thm}
\begin{proof}
It follows from the discussion in Section \ref{singsize} that any
rescaling limit $(\m_\infty,g_\infty(t))$ as in Theorem
\ref{limsolthm} around $p\in\m \setminus\Sigma$ is flat, i.e.
$(\m_\infty,g_{\infty}(t))$ is isometric to the Gaussian soliton
$(\R^n,g_{\R^n}).$ Hence one easily computes $\theta_{p,T}=1.$

Conversely, let $\theta_{p,T}$ be the density of $p\in \m$ and let
$(\m_\infty,g_\infty(t),l^\infty(t))$ be the rescaling limit flow
around $(p,T).$ It is a normalized gradient shrinking soliton in
canonical form with constant formal reduced volume
\begin{equation*}
\theta_{p,T}=\int_{\m_\infty} (4\pi(T-t))^{-\frac{n}{2}}
e^{-l^\infty(t)} dvol_{g_\infty(t)}
\end{equation*}
for any $t.$ Note that in the proof of Theorem \ref{limsolthm} we
can use $l_{p,T}$ instead of $l_{p,T,\{t_{k_l}\}}$ as it only
requires the estimates from Lemma \ref{lest} as well as the formal equalities
on a constant reduced volume. Now we can employ \cite[Corollary 1.1
(3)]{Yokota} to conclude that there exists $\eta>0$ (only depending
on $n$) such that if $\theta_{p,T}>1-\eta$, then the limit flow is
the Gaussian soliton. In particular, $\theta_{p,T}>1-\eta$ implies
$p\in \m \setminus \Sigma$ by Theorems \ref{mainthm1} and \ref{mainthm2}.
\end{proof}

%------------------------------------------------
\makeatletter
\def\@listi{%
  \itemsep=0pt
  \parsep=1pt
  \topsep=1pt}
\makeatother
\textheight=230mm
{\fontsize{10}{11}\selectfont
}

Joerg Enders\\
{\sc Institut f\"ur Mathematik, Universit\"at Potsdam, 14469 Potsdam, Germany}

Reto M\"uller\\
{\sc Scuola Normale Superiore di Pisa, 56126 Pisa, Italy}

Peter M.\ Topping\\
{\sc Mathematics Institute, University of Warwick, Coventry,
CV4 7AL, UK}

\end{document}